\documentclass[11pt,a4paper]{article}

\usepackage{amsthm}
\usepackage{amsmath}
\usepackage{amssymb}
\usepackage{mathrsfs}
\usepackage{geometry}
\usepackage{graphicx}
\usepackage{amsfonts}
\usepackage{epstopdf}
\usepackage{enumerate}
\usepackage{hyperref}
\usepackage{subcaption}

\allowdisplaybreaks

\newcommand{\veps}{\varepsilon}

\newcommand{\md}{\mathrm{d}}

\newcommand{\R}{\mathbb{R}}

\newcommand{\Rmnum}[1]{\uppercase\expandafter{\romannumeral#1}} 

\newcommand{\bbZ}{\mathbb{Z}}

\newcommand{\calE}{\mathcal{E}}
\newcommand{\calF}{\mathcal{F}}
\newcommand{\calG}{\mathcal{G}}
\newcommand{\calH}{\mathcal{H}}

\newcommand{\calP}{\mathcal{P}}

\newcommand{\rmC}{\mathrm{C}}

\newcommand{\rmH}{\mathrm{H}}

\newcommand{\rmT}{\mathrm{T}}

\newcommand{\rmdim}{\mathrm{dim}}

\newcommand{\myset}[1]{\left\{#1\right\}}

\newcommand{\mynorm}[1]{\lVert#1\rVert}
\newcommand{\mytilde}[1]{\widetilde{#1}}

\def\Xint#1{\mathchoice
{\XXint\displaystyle\textstyle{#1}}%
{\XXint\textstyle\scriptstyle{#1}}%
{\XXint\scriptstyle\scriptscriptstyle{#1}}%
{\XXint\scriptscriptstyle\scriptscriptstyle{#1}}%
\!\int}
\def\XXint#1#2#3{{\setbox0=\hbox{$#1{#2#3}{\int}$ }
\vcenter{\hbox{$#2#3$ }}\kern-.6\wd0}}

\def\dashint{\Xint-}

\newtheorem{mythm}{Theorem}[section]
\newtheorem{myprop}[mythm]{Proposition}
\newtheorem{mylem}[mythm]{Lemma}

\newtheorem{myrmk}[mythm]{Remark}
\newtheorem{myrmks}[mythm]{Remarks}
\newtheorem{mydef}[mythm]{Definition}

\renewcommand{\thefootnote}{}

\AtEndDocument{
\bigskip{\footnotesize%
  \textsc{Institut f\"ur Angewandte Mathematik, Universit\"at Bonn, Endenicher Allee 60, 53115 Bonn, Germany} \par
  \textit{E-mail address}: \texttt{yangmengqh@gmail.com}
}}

\begin{document}

\title{{\large{Korevaar-Schoen spaces on Sierpi\'nski carpets}}}
\author{Meng Yang}
\date{}

\maketitle

\abstract{We prove that certain $L^p$-regularity functional inequality holds on \emph{generalized Sierpi\'nski carpets}. This gives an affirmative answer to an open question raised by Fabrice Baudoin. Our technique originates from an old idea of Alf Jonsson in 1996.}

\footnote{\textsl{Date}: \today}
\footnote{The author was very grateful to Jin Gao, Zhenyu Yu and Junda Zhang for introducing their paper \cite{GYZ23}. He was also very grateful to Fabrice Baudoin, Mathav Murugan and Ryosuke Shimizu for helpful discussions. Part of the work was carried out when the author was attending the workshop \emph{Potential theory and random walks in metric spaces} held in Okinawa, Japan. He acknowledged the organizers for their invitation.}

\renewcommand{\thefootnote}{\arabic{footnote}}
\setcounter{footnote}{0}

\section{Introduction}

As generalizations of the standard Sobolev spaces $W^{1,p}(\R^d)$ with $p$-energy $\int_{\R^d}|\nabla f(x)|^p\md x$, some developments were recently given about $p$-energy on fractal spaces and metric measure spaces, in particular, the spaces where upper gradients are not available. Certain $L^p$-regularity functional inequalities, such as the condition $\calP(p,\delta)$ which is given as follows, play an important role.

Let $(K,d,\mu)$ be a doubling metric measure space. For $p\ge1$, $\delta>0$, $f\in L^p(K;\mu)$, $r>0$, define
$$E_{p,\delta}(f,r)=\frac{1}{r^{p\delta}}\int_K\left(\dashint_{B(x,r)}|f(x)-f(y)|^p\md\mu(y)\right)\md\mu(x),$$
where $\dashint_A=\frac{1}{\mu(A)}\int_A$ for a Borel measurable set $A$ with $\mu(A)\in(0,+\infty)$. The Korevaar-Schoen space $\mathcal{KS}^{p,\delta}(K)$ is defined as
$$\mathcal{KS}^{p,\delta}(K)=\myset{f\in L^p(K;\mu):\varlimsup_{r\downarrow0}E_{p,\delta}(f,r)<+\infty}.$$
We say that the condition $\calP(p,\delta)$ holds if there exists some positive constant $C$ such that for any $f\in \mathcal{KS}^{p,\delta}(K)$
$$\sup_{r>0}E_{p,\delta}(f,r)\le C\varliminf_{r\downarrow0}E_{p,\delta}(f,r).$$

We call the type of the condition, ``$\sup\ldots\lesssim\varliminf\ldots$", that the \emph{global} supremum term can be bounded by the \emph{local} limit infimum term, the \emph{weak monotonicity type condition} (following \cite{GYZ23}). The condition $\calP(p,\delta)$ was introduced in \cite[Definition 4.5]{Bau22a}. Another weak monotonicity type condition, denoted as $(P_{p,\theta})$, which involves the heat kernel and was introduced in \cite[Definition 6.7]{ABCRST1}, is as follows. There exists some positive constant $C$ such that for any $f$ in the Besov type space $\mathbf{B}^{p,\theta}(K)$
\begin{align*}
&\sup_{t>0}\frac{1}{t^{p\theta}}\int_K\int_K|f(x)-f(y)|^pp_t(x,y)\md\mu(y)\md\mu(x)\\
&\le C\varliminf_{t\downarrow0}\frac{1}{t^{p\theta}}\int_K\int_K|f(x)-f(y)|^pp_t(x,y)\md\mu(y)\md\mu(x).
\end{align*}

If the heat kernel satisfies two-sided Gaussian or sub-Gaussian estimate with walk dimension $d_w$, then it was proved in \cite[Theorem 1.7]{GYZ23} that $\calP(p,\delta)$ is equivalent to $(P_{p,\delta/{d_w}})$ ($(NE)$ and $(KE)$ in their notation).

The weak monotonicity type conditions come naturally in the critical exponent phenomenon of Korevaar-Schoen spaces and Besov spaces, see \cite[Lemma 4.7]{Bau22a}.

The weak monotonicity type conditions have many significant applications. First, they can give fruitful Sobolev inequalities, see \cite[Theorem 1.1]{ABCRST1} with $(P_{p,\theta})$ and \cite[Subsection 4.3]{Bau22a} with $\calP(p,\delta)$. Second, they can give generalizations of the well-known Bourgain-Brezis-Mironescu (BBM) convergence, see \cite[Theorem 1.3]{GYZ23} with $(P_{p,\theta})$. Third, for $p=2$, they are related to the construction of local regular Dirichlet forms, see \cite[Assumptions 1 (A$2^*$)]{KS05}. Fourth, for $p=1$, they are related to BV functions, see \cite[Remark 4.5]{ABCRST2}, which is $\calP(1,1)$, for the Gaussian case, and \cite[Theorem 4.9]{ABCRST3}, which is $\calP(1,d_W-\kappa)$, for the sub-Gaussian case.

A pretty natural and interesting question is to check the validity of $\calP(p,\delta)$\footnote{Implicitly, we  also mean that $\mathcal{KS}^{p,\delta}(K)$ contains non-constant functions, otherwise the result is trivial. Indeed, by \cite[Lemma 4.7]{Bau22a}, for fixed $p\ge1$, there exists at most one $\delta>0$ such that $\calP(p,\delta)$ holds and $\mathcal{KS}^{p,\delta}(K)$ contains non-constant functions simultaneously.} on typical fractal spaces. Recently, it was proved in \cite[Section 6]{Bau22a} that $\calP(p,\delta_p)$ holds on the Vicsek set and the Sierpi\'nski gasket, here $\delta_p$ is a parameter depending on $p$, and in \cite[Section 4]{GYZ23} that $\calP(p,\delta_p)$ holds on homogeneous p.c.f. self-similar sets. The above fractal spaces are finitely ramified, for infinitely ramified fractals, such as the Sierpi\'nski carpet, it was left open and raised by Fabrice Baudoin in \cite[Page 2 and Page 19]{Bau22a} that

\vspace{5pt}

\noindent
\emph{``For some other spaces like the Sierpi\'nski carpet the validity of (1) is still an open
question."}

\vspace{5pt}

\noindent
and

\vspace{5pt}

\noindent
\emph{``The validity or not of the property $\calP(p,\alpha_p)$ would be an interesting question to settle then since the carpet is an example of infinitely ramified fractal, making its geometry very different from the two examples treated below."}

\vspace{5pt}

The main purpose of this paper is to show that $\calP(p,\beta_p/p)$ holds with $\mathcal{KS}^{p,\beta_p/p}(K)$ containing non-constant functions on \emph{generalized Sierpi\'nski carpets} for any $p>\mathrm{dim}_{\mathrm{ARC}}(K,d)$, where $\mathrm{dim}_{\mathrm{ARC}}(K,d)\ge1$ is the Ahlfors regular conformal dimension, and $\beta_p$ is a parameter depending on $p$, which gives an affirmative answer to the above question. Recently, \cite{MS23} proved that $\calP(p,\delta_p)$ holds for any $p>1$ on the standard Sierpi\'nski carpet. Similar to the proofs in \cite{Bau22a,GYZ23}, our proof also relies on the construction of $p$-energy, of course on generalized Sierpi\'nski carpets, which was recently given by \cite{Shi24}. The construction of $p$-energy on fractal spaces or metric measure spaces was firstly given by \cite{HPS04} on Sierpi\'nski gasket type fractals. Recently, it was given by \cite{CGQ22} on p.c.f. self-similar sets, \cite{Shi24} on generalized Sierpi\'nski carpets, and \cite{Kig21book} on metric measure spaces.

This paper is organized as follows. In Section \ref{sec_result}, we introduce basic definitions to state our main result. In Section \ref{sec_proof}, we give the proof of our main result.

\section{Result}\label{sec_result}

In this section, we introduce basic definitions to state our main result.

Let $(K,d)$ be a metric space. Denote $B(x,r)=\{y\in K:d(x,y)<r\}$ for $x\in K$, $r\in(0,+\infty)$. Let $\alpha\in(0,+\infty)$. We say that a Borel measure $\mu$ on $K$ is $\alpha$-Ahlfors regular if there exists some positive constant $C_\mathrm{AR}$ such that
$$\frac{1}{C_{\mathrm{AR}}}r^\alpha\le\mu(B(x,r))\le C_{\mathrm{AR}}r^{\alpha}$$
for any $x\in K$, $r\in(0,\mathrm{diam}(K,d))$, where $\mathrm{diam}(K,d)=\sup\{d(x,y):x,y\in K\}$ is the diameter of the metric space $(K,d)$. If such a Borel measure $\mu$ exists, then by the well-known Frostman lemma (see \cite[Page 112]{Mat95book} and \cite[Page 60]{Fal03book}), the Hausdorff dimension $\mathrm{dim}_{\mathrm{H}}(K,d)$ of the metric space $(K,d)$ is $\alpha$, and we say that the metric $d$ or the metric space $(K,d)$ is ($\alpha$-)Ahlfors regular.

Ahlfors regular conformal dimension was introduced implicitly by Bourdon and Pajot \cite{BP03}. We give the definition as follows.

\begin{mydef}\label{def_ARC}
Let $(K_1,d_1)$, $(K_2,d_2)$ be two metric spaces. We say that $(K_1,d_1)$ is quasisymmetric to $(K_2,d_2)$ if there exist a homeomorphism $f:K_1\to K_2$ and a homeomorphism $\eta:[0,+\infty)\to[0,+\infty)$ such that
$$\frac{d_2(f(x),f(a))}{d_2(f(x),f(b))}\le\eta\left(\frac{d_1(x,a)}{d_1(x,b)}\right)$$
for any $x,a,b\in K_1$ with $x\ne b$. It is easy to verify that quasisymmetry is an equivalence relation. Let $(K,d)$ be a metric space. Define the conformal gauge $\calG(K,d)$ as the set of all metric spaces $(K_0,d_0)$ which are quasisymmetric to $(K,d)$. The Ahlfors regular conformal dimension $\mathrm{dim}_{\mathrm{ARC}}(K,d)$ is defined as
$$\mathrm{dim}_{\mathrm{ARC}}(K,d)=\inf\left\{\mathrm{dim}_\mathrm{H}(K_0,d_0):(K_0,d_0)\in\calG(K,d)\text{ is Ahlfors regular}\right\}.$$
\end{mydef}

\begin{myrmks}
\noindent
\begin{enumerate}[(1)]
\item Let $(K,d)$ be a separable metric space. Then its topological dimension $\rmdim_\rmT(K,d)$ has the following characterization.
$$\rmdim_\rmT(K,d)=\inf\left\{\rmdim_\rmH(K_0,d_0):K_0\text{ is homeomorphic to }K\right\}\in\{-1,0,1,\ldots\}.$$
The above characterization was firstly obtained by Szpilrajn \cite{Szp37} in 1937, see also \cite[VII.4]{HW41book}. Moreover, if $(K,d)$ is locally compact, then $\rmdim_\rmT(K,d)\in\{-1,0\}$ if and only if $K$ is empty or totally disconnected, see \cite[Theorem 1.4.5]{Eng78book}. Hence for a locally compact separable connected non-empty metric space $(K,d)$, we have $\rmdim_\rmT(K,d)\in\{1,2,\ldots\}$.

Recall that the conformal dimension $\rmdim_\rmC(K,d)$ of a metric space $(K,d)$ is defined as
$$\rmdim_\rmC(K,d)=\inf\left\{\rmdim_\rmH(K_0,d_0):(K_0,d_0)\in\calG(K,d)\right\}.$$

Hence for a locally compact separable connected Ahlfors regular non-empty metric space $(K,d)$, we have
$$1\le\rmdim_\rmT(K,d)\le\rmdim_\rmC(K,d)\le\rmdim_{\mathrm{ARC}}(K,d)\le\rmdim_\rmH(K,d).$$
\item For the standard Sierpi\'nski carpet, see Figure \ref{fig_SC}, we have
$$1+\frac{\log2}{\log3}\le\rmdim_\mathrm{ARC}(K,d)<\frac{\log8}{\log3},$$
see \cite[Example 4.3.1]{MT10book}, or \cite{Tys00a} and \cite{KL04}.
\end{enumerate}
\end{myrmks}

Let $\mu$ be a doubling Borel measure on $(K,d)$, that is, there exists some positive constant $C_{\mathrm{VD}}\ge1$ such that $\mu(B(x,2r))\le C_{\mathrm{VD}}\mu(B(x,r))$ for any $x\in K$, $r\in(0,+\infty)$. Then it is easy to show that for any $x,y\in K$, $r,R\in(0,+\infty)$ with $r\le R$, we have
$$\frac{\mu(B(x,R))}{\mu(B(y,r))}\le C_{\mathrm{VD}}\left(\frac{d(x,y)+R}{r}\right)^{\frac{\log C_{\mathrm{VD}}}{\log2}},$$
see \cite[Proposition 3.2]{GHL09}.

For $p\in[1,+\infty)$, $\delta\in(0,+\infty)$, $f\in L^p(K;\mu)$, $r\in(0,+\infty)$, define
$$E_{p,\delta}(f,r)=\frac{1}{r^{p\delta}}\int_K\left(\dashint_{B(x,r)}|f(x)-f(y)|^p\md\mu(y)\right)\md\mu(x),$$
where $\dashint_A=\frac{1}{\mu(A)}\int_A$ for a Borel measurable set $A$ with $\mu(A)\in(0,+\infty)$.
The Korevaar-Schoen space $\mathcal{KS}^{p,\delta}(K)$ and the Besov-Lipschitz space $\mathcal{B}^{p,\delta}(K)$ are defined as follows.
\begin{align*}
\mathcal{KS}^{p,\delta}(K)&=\myset{f\in L^p(K;\mu):\varlimsup_{r\downarrow0}E_{p,\delta}(f,r)<+\infty},\\
\mathcal{B}^{p,\delta}(K)&=\myset{f\in L^p(K;\mu):\sup_{r>0}E_{p,\delta}(f,r)<+\infty}.
\end{align*}
By \cite[Lemma 3.2]{Bau22a}, we have $\mathcal{KS}^{p,\delta}(K)=\mathcal{B}^{p,\delta}(K)$. It is easy to verify that if $\delta_1,\delta_2\in(0,+\infty)$ satisfy $\delta_1<\delta_2$, then $\mathcal{KS}^{p,\delta_1}(K)\supseteq\mathcal{KS}^{p,\delta_2}(K)$, hence $\mathcal{KS}^{p,\delta}(K)$ may be trivial which consists of only constant functions if $\delta\in(0,+\infty)$ is too large.

By the doubling property, for any $c\in(0,+\infty)$, $a\in(1,+\infty)$, $n\in\bbZ$, for any $r\in(ca^{-(n+1)},ca^{-n}]$, we have
$$\frac{1}{C_{\mathrm{VD}}a^{\frac{\log C_{\mathrm{VD}}}{\log2}+p\delta}}E_{p,\delta}(\cdot,ca^{-(n+1)})\le E_{p,\delta}(\cdot,r)\le\left(C_{\mathrm{VD}}a^{\frac{\log C_{\mathrm{VD}}}{\log2}+p\delta}\right)E_{p,\delta}(\cdot,ca^{-n}).$$
Hence
\begin{align*}
&\mathcal{KS}^{p,\delta}(K)=\myset{f\in L^p(K;\mu):\varlimsup_{n\to+\infty}E_{p,\delta}(f,ca^{-n})<+\infty}\\
&=\mathcal{B}^{p,\delta}(K)=\myset{f\in L^p(K;\mu):\sup_{n\ge1}E_{p,\delta}(f,ca^{-n})<+\infty}.
\end{align*}

\begin{mydef}(\cite[Definition 4.5]{Bau22a})
Let $p\ge1$, $\delta>0$. We say that the condition $\calP(p,\delta)$ holds if there exists some positive constant $C$ such that for any $f\in\mathcal{KS}^{p,\delta}(K)$, we have
$$\sup_{r>0}E_{p,\delta}(f,r)\le C\varliminf_{r\downarrow0}E_{p,\delta}(f,r).$$
\end{mydef}

By \cite[Lemma 4.7]{Bau22a}, if $\calP(p,\delta)$ holds with $\mathcal{KS}^{p,\delta}(K)$ containing non-constant functions, then $\mathcal{KS}^{p,\delta_0}(K)$ consists of only constant functions for any $\delta_0>\delta$. Indeed, for fixed $p\ge1$, there exists at most one $\delta>0$ such that $\calP(p,\delta)$ holds and $\mathcal{KS}^{p,\delta}(K)$ contains non-constant functions simultaneously.

Let $D\ge2$, $a\ge3$ be two integers. For any $i\in\{0,1,\ldots,a-1\}^D$, let $f_i:\R^D\to\R^D$ be given by $x\mapsto(x+i)/a$, $x\in\R^D$. Let $S\subsetneqq\{0,1,\ldots,a-1\}^D$ be a non-empty set, then there exists a unique non-empty compact set $K$ in $\R^D$ such that $K=\bigcup_{i\in S}f_i(K)$. Let $F_i=f_i|_K$ and $\mathrm{GSC}(D,a,S)=(K,S,\{F_i\}_{i\in S})$. Let $Q_0=[0,1]^D$ and $Q_1=\cup_{i\in S}f_i(Q_0)$, then $Q_1\subsetneqq Q_0$. The following definition of generalized Sierpi\'nski carpets is from \cite[Subsection 2.2]{BBKT10}.

\begin{mydef}(Generalized Sierpi\'nski carpets)\label{def_GSC}
We say that $\mathrm{GSC}(D,a,S)$ is a generalized Sierpi\'nski carpet if the following four conditions are satisfied.
\begin{enumerate}[(1)]
\item\label{def_GSC1} (Symmetry) For any isometry $f:Q_0\to Q_0$, we have $f(Q_1)=Q_1$.
\item\label{def_GSC2} (Connectedness) $\mathrm{Int}(Q_1)$ is connected, where $\mathrm{Int}(\cdot)$ is the interior of a subset of $\R^D$ in the Euclidean topology.
\item\label{def_GSC3} (Non-diagonality) For any $i_1,\ldots,i_D\in\{0,1,\ldots,a-2\}$, if $\mathrm{Int}\left(Q_1\cap\Pi_{k=1}^D\left[\frac{i_k}{a},\frac{i_k+2}{a}\right]\right)$ is non-empty, then it is connected.
\item\label{def_GSC4} (Borders included) $[0,1]\times\{0\}^{D-1}\subseteq Q_1$.
\end{enumerate}
\end{mydef}

The standard Sierpi\'nski carpet is indeed the generalized Sierpi\'nski carpet
$$\mathrm{GSC}(2,3,\{0,1,2\}^2\backslash\{(1,1)\}),$$
see Figure \ref{fig_SC}. In this paper, we always assume that $\mathrm{GSC}(D,a,S)=(K,S,\{F_i\}_{i\in S})$ is a generalized Sierpi\'nski carpet.

\begin{figure}[ht]
\centering
\includegraphics[width=0.5\textwidth]{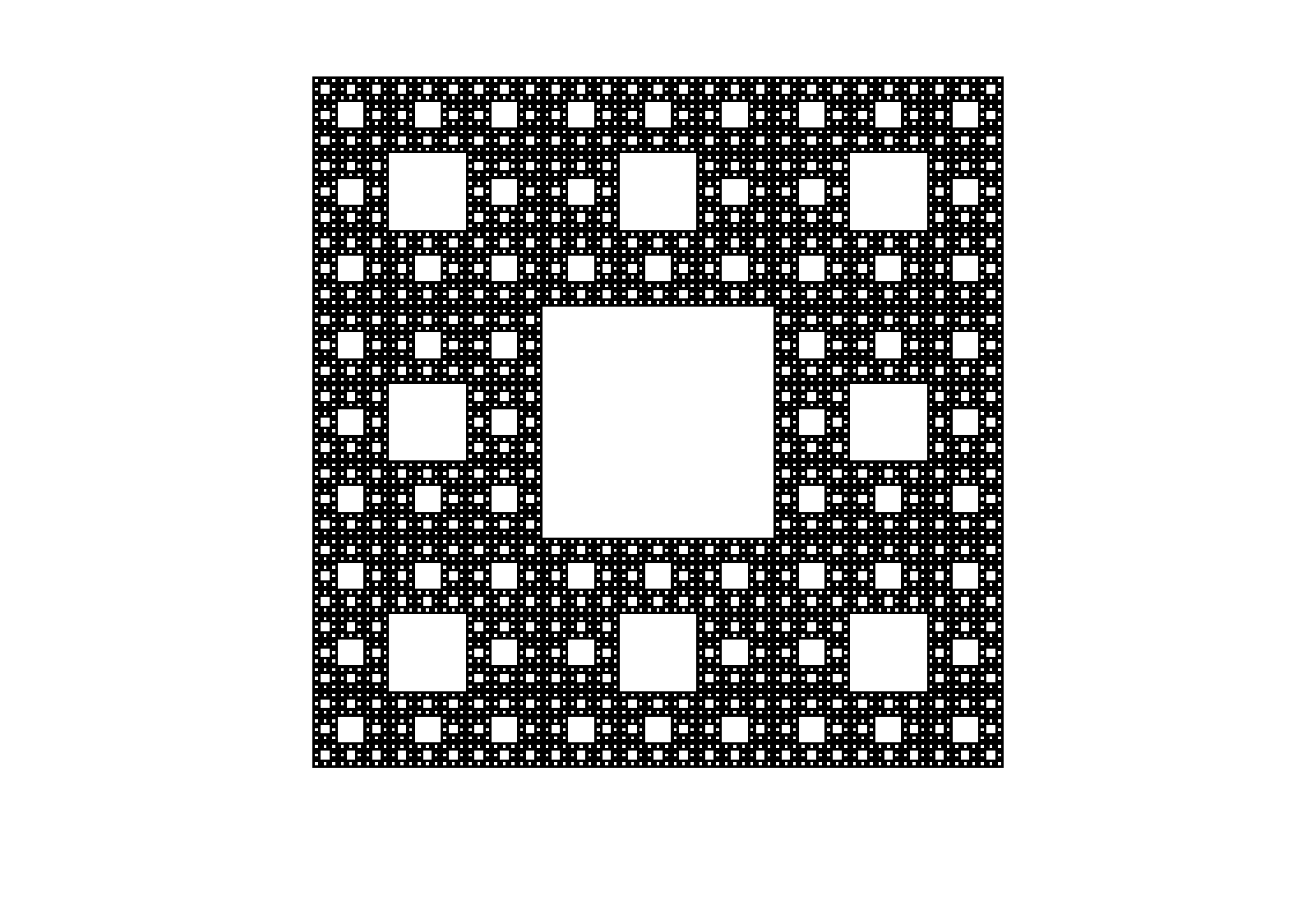}
\caption{The standard Sierpi\'nski carpet}\label{fig_SC}
\end{figure}

Let $N_*=\#S$ and $\alpha=\log N_*/\log a$. Let $d:K\times K\to[0,+\infty)$ be the standard Euclidean metric. Let $\mu$ be the self-similar probability measure on $K$ with weight $(1/N_*,\ldots,1/N_*)$, that is, $\mu$ is the unique Borel probability measure on $K$ satisfying $\mu=\sum_{i\in S}(1/N_*)(\mu\circ F_i^{-1})$. By \cite[Theorem 1.5.7, Proposition 1.5.8]{Kig01book}, we have the Hausdorff dimension $\mathrm{dim}_\rmH(K,d)=\alpha$, the $\alpha$-Hausdorff measure $\calH^\alpha(K)\in(0,+\infty)$, $\mu=(1/\calH^\alpha(K))\calH^\alpha$ is the normalized $\alpha$-Hausdorff measure and $\mu$ is $\alpha$-Ahlfors regular, hence $(K,d)$ is $\alpha$-Ahlfors regular.

Let $W_0=\{\emptyset\}$ and $W_n=\{w=w_1\ldots w_n:w_k\in S\text{ for any }k=1,\ldots,n\}$ for any $n\ge1$. Let $W_\infty=\{w=w_1w_2\ldots:w_k\in S\text{ for any }k\ge1\}$.

For any $w^{(1)}=w^{(1)}_1\ldots w^{(1)}_n\in W_n$ and $w^{(2)}=w^{(2)}_1\ldots w^{(2)}_m\in W_m$, denote
$$w^{(1)}w^{(2)}=w^{(1)}_1\ldots w^{(1)}_nw^{(2)}_1\ldots w^{(2)}_m\in W_{n+m}$$
as the concatenation of $w^{(1)}$ and $w^{(2)}$.

For any $w^{(1)}=w^{(1)}_1\ldots w^{(1)}_n\in W_n$ and $w^{(2)}=w^{(2)}_1w^{(2)}_2\ldots\in W_\infty$, denote
$$w^{(1)}w^{(2)}=w^{(1)}_1\ldots w^{(1)}_nw^{(2)}_1w^{(2)}_2\ldots\in W_{\infty}$$
as the concatenation of $w^{(1)}$ and $w^{(2)}$.

For any $n\ge0$, $w=w_1\ldots w_n\in W_n$, let $F_w=F_{w_1}\circ\ldots\circ F_{w_n}$ and $K_w=F_w(K)$, here we use the convention that $F_\emptyset=\mathrm{Id}$ is the identity map. For any $w=w_1w_2\ldots\in W_\infty$, let $K_w=\cap_{n\ge1}K_{w_1\ldots w_n}$, then $K_w\subseteq K$ contains exactly one point. For any $n\ge0$, $w\in W_n$, $x\in K_w$, there exists $v\in W_\infty$ which is not necessarily unique such that $\{x\}=K_{wv}$.

For any $n\ge1$, let $G_n=(W_n,E_n)$ be the undirected graph with vertex set $W_n$ and edge set $E_n$ given by
$$E_n=\myset{(w,v):w,v\in W_n,w\ne v,K_w\cap K_v\ne\emptyset}.$$
By Definition \ref{def_GSC} (\ref{def_GSC2}), we have $G_n$ is connected, that is, for any $w,v\in W_n$ with $w\ne v$, there exist $k\ge1$ and $w^{(0)},w^{(1)},\ldots,w^{(k)}\in W_n$ with $w^{(0)}=w$, $w^{(k)}=v$ such that $(w^{(i)},w^{(i+1)})\in E_n$ for any $i=0,\ldots,k-1$.

Let
\begin{equation}\label{eqn_L*}
L_*=\sup_{n\ge1}\sup_{w\in W_n}\#\myset{v\in W_n:(w,v)\in E_n},
\end{equation}
then $L_*\le3^D-1$. Let $c\in(0,+\infty)$ be a constant satisfying that for any $n\ge1$, for any $(w,v)\in E_n$, we have
\begin{equation}\label{eqn_smallc}
\sup_{x,y\in K_w} d(x,y)<ca^{-n}\quad\text{and}\quad\sup_{x\in K_w,y\in K_v}d(x,y)<ca^{-n}.
\end{equation}
The existence of such $c$ follows from \cite[Lemma 2.13]{Shi24}.

Let $p\in(0,+\infty)$ and $G=(V,E)$ an undirected graph with vertex set $V$ and edge set $E$, then the $p$-energy on $G$ is defined as
$$\calE_p^G(f)=\sum_{(w,v)\in E}|f(w)-f(v)|^p\text{ for }f\in l(V),$$
where $l(Z)$ is the set of all real-valued functions on the set $Z$.

For $p\in[1,+\infty)$, $n\ge0$, let $M_n:L^p(K;\mu)\to l(W_n)$ be given by
$$M_nf(w)=\int_{K}\left(f\circ F_w\right)\md\mu=\dashint_{K_w}f\md\mu=\frac{1}{\mu(K_w)}\int_{K_w}f\md\mu=N_*^n\int_{K_w}f\md\mu$$
for any $f\in L^p(K;\mu)$, $w\in W_n$.

For any $p>\mathrm{dim}_{\mathrm{ARC}}(K,d)$, let $\rho_p\in(1,+\infty)$ be the parameter given by \cite[Theorem 3.4, Theorem 4.9]{Shi24}. Indeed, the existence of $\rho_p\in(0,+\infty)$ follows from the submultiplicative result \cite[Theorem 3.4]{Shi24}, the bound $\rho_p>1$ is equivalent to the bound $p>\mathrm{dim}_{\mathrm{ARC}}(K,d)$, which follows from \cite[Theorem 1.3]{Car13} or \cite[Theorem 4.6.9, Theorem 4.7.6]{Kig20LNM}. Let $\beta_p=\log(N_*\rho_p)/\log a$, then $\beta_p>\alpha$ for any $p>\mathrm{dim}_{\mathrm{ARC}}(K,d)$. For $n\ge1$, let $\mytilde{\calE}_p^{G_n}(\cdot)=\rho_p^n\calE_p^{G_n}(\cdot)$ be the rescaled $p$-energy.

The main result of this paper is that on generalized Sierpi\'nski carpets, for any $p>\mathrm{dim}_{\mathrm{ARC}}(K,d)$, we have $\calP(p,\beta_p/p)$ holds.

We list the results about $p$-energy on generalized Sierpi\'nski carpets as follows.

\begin{mylem}(\cite[Theorem 2.21]{Shi24})
Assume that $p>\mathrm{dim}_{\mathrm{ARC}}(K,d)$. Let
$$\calF_p=\myset{f\in L^p(K;\mu):\sup_{n\ge1}\mytilde{\calE}_p^{G_n}(M_nf)<+\infty},$$
and
$$\mynorm{f}_{\calF_p}=\mynorm{f}_{L^p(K;\mu)}+\left(\sup_{n\ge1}\mytilde{\calE}_p^{G_n}(M_nf)\right)^{1/p}.$$
Then $(\calF_p,\mynorm{\cdot}_{\calF_p})$ is a reflexive separable Banach space which is continously embedded in the H\"older space
$$C^{0,(\beta_p-\alpha)/p}(K)=\myset{f\in C(K):\sup_
{\mbox{\tiny
$
\begin{subarray}{c}
x,y\in K\\
x\ne y
\end{subarray}
$
}}
\frac{|f(x)-f(y)|}{d(x,y)^{(\beta_p-\alpha)/p}}<+\infty}.$$
Moreover, $\calF_p$ is uniformly dense in $C(K)$.
\end{mylem}

\begin{mylem}(\cite[Theorem 2.22, Theorem 5.15]{Shi24})
Assume that $p>\mathrm{dim}_{\mathrm{ARC}}(K,d)$. Then
$$\calF_p=\mathcal{KS}^{p,\beta_p/p}(K)=\mathcal{B}^{p,\beta_p/p}(K).$$
\end{mylem}

\begin{myrmk}
By the above two results, $\mathcal{KS}^{p,\beta_p/p}(K)$ contains ``many" ``well-hehaved" continuous functions which are particularly non-constant.
\end{myrmk}

Before giving our formal main result, we need to do some reductions.

First, since $(K,d)$ is bounded, we have $E_{p,\delta}(\cdot,r)\le E_{p,\delta}(\cdot,\mathrm{diam}(K,d))$ for any $r>\mathrm{diam}(K,d)$. Since $(K,d)$ satisfies the chain condition (see \cite[Definition 3.4]{GHL03}), by a similar argument to the proof of \cite[Corollary 2.2]{Yan18}, for any $r_1,r_2\in(0,+\infty)$ with $r_1<r_2$, there exists some positive constant $C$ depending on $r_1$, $r_2$ such that $E_{p,\delta}(\cdot,r_2)\le CE_{p,\delta}(\cdot,r_1)$. Therefore, we only need to show that for some fixed $r_0\in(0,+\infty)$, there exists some positive constant $C$ which may depend on $r_0$ such that for any $f\in\mathcal{KS}^{p,\beta_p/p}(K)$, we have
$$\sup_{r\in(0,r_0)}E_{p,\beta_p/p}(f,r)\le C\varliminf_{r\downarrow0}E_{p,\beta_p/p}(f,r).$$

Second, for $n\ge1$, $f\in L^p(K;\mu)$, let
$$A^{(n)}_{p,\beta_p}(f)=a^{(\alpha+\beta_p)n}\int_K\int_{B(x,ca^{-n})}|f(x)-f(y)|^p\md\mu(y)\md\mu(x),$$
where $c$ is the constant given by Equation (\ref{eqn_smallc}). Since $\mu$ is $\alpha$-Ahlfors regular, it is easy to verify that
$$\frac{1}{C_{\mathrm{AR}}c^{\alpha+\beta_p}}A^{(n)}_{p,\beta_p}\le E_{p,\beta_p/p}(\cdot,ca^{-n})\le\frac{C_{\mathrm{AR}}}{c^{\alpha+\beta_p}}A^{(n)}_{p,\beta_p}.$$
Hence in terms of $A^{(n)}_{p,\beta_p}$, we have
\begin{align*}
\calF_p&=\mathcal{KS}^{p,\beta_p/p}(K)=\myset{f\in L^p(K;\mu):\varlimsup_{n\to+\infty}A^{(n)}_{p,\beta_p}(f)<+\infty}\\
&=\mathcal{B}^{p,\beta_p/p}(K)=\myset{f\in L^p(K;\mu):\sup_{n\ge1}A^{(n)}_{p,\beta_p}(f)<+\infty}.
\end{align*}

Combining the above two reductions, we state the main result of this paper as follows.

\begin{mythm}\label{thm_main}
Assume that $p>\mathrm{dim}_{\mathrm{ARC}}(K,d)$. Then there exists some positive constant $C$ such that for any $f\in\mathcal{KS}^{p,\beta_p/p}(K)$, we have
$$\sup_{n\ge1}A^{(n)}_{p,\beta_p}(f)\le C\varliminf_{n\to+\infty}A^{(n)}_{p,\beta_p}(f)$$

\end{mythm}

From now on, we always assume that $p>\mathrm{dim}_{\mathrm{ARC}}(K,d)$.

\section{Proof}\label{sec_proof}

In this section, we give the proof of our main result Theorem \ref{thm_main}. Before that, we collect some necessary technical results from \cite{Shi24} as follows.

First, we have the following weak monotonicity result.

\begin{mylem}\label{lem_WM}(\cite[Corollary 4.16]{Shi24})
There exists some positive constant $C_{\mathrm{WM}}$ such that
$$\mytilde{\calE}_{p}^{G_{n}}(M_{n}f)\le C_{\mathrm{WM}}\mytilde{\calE}_{p}^{G_{n+m}}(M_{n+m}f)$$
for any $n,m\ge1$, $f\in L^p(K;\mu)$. In particular
$$\sup_{n\ge1}\mytilde{\calE}_{p}^{G_{n}}(M_{n}f)\le C_{\mathrm{WM}}\varliminf_{n\to+\infty}\mytilde{\calE}_{p}^{G_{n}}(M_{n}f)$$
for any $f\in L^p(K;\mu)$.
\end{mylem}

Second, we have the following $(p,p)$-Poincar\'e inequality.

\begin{mylem}\label{lem_lem512}(\cite[Lemma 5.12]{Shi24})
There exists some positive constant $C_{\mathrm{PI}\text{-}\mathrm{KS}}$ such that
$$a^{{\beta_p}n}\sum_{w\in W_n}\int_{K_w}|f-M_nf(w)|^p\md\mu\le C_{\mathrm{PI}\text{-}\mathrm{KS}}\varliminf_{l\to+\infty}\mytilde{\calE}_{p}^{G_{n+l}}(M_{n+l}f)$$
for any $n\ge1$, $f\in\calF_p$.
\end{mylem}

Third, we have the following H\"older regularity result.

\begin{mylem}\label{lem_lem513}(\cite[Lemma 5.13]{Shi24})
There exists some positive constant $C_{5.13}$ such that
$$|f(x)-f(y)|^p\le C_{5.13}d(x,y)^{\beta_p-\alpha}\sup_{n\ge1}A^{(n)}_{p,\beta_p}(f)$$
for any $x,y\in K$, for any $f\in\calF_p$.
\end{mylem}

Fourth, we have the following technical inequality.

\begin{mylem}\label{lem_eqn519}(\cite[Equation (5.19)]{Shi24})
There exists some positive constant $c_{5.19}$ such that
\begin{align*}
&A^{(n)}_{p,\beta_p}(f)\le c_{5.19}\left(\mytilde{\calE}_p^{G_n}(M_nf)+a^{\beta_pn}\sum_{w\in W_n}\int_{K_w}|f-M_nf(w)|^p\md\mu\right)
\end{align*}
for any $n\ge1$, $f\in L^p(K;\mu)$.
\end{mylem}

We divide the proof of our main result into the following four propositions. The idea is to relate $A^{(n)}_{p,\beta_p}$ to $\mytilde{\calE}_{p}^{G_{n}}(M_n\cdot)$, in some proper sense that each of them can be ``bounded" by the other, such that the weak monotonicity result Lemma \ref{lem_WM} for $\mytilde{\calE}_{p}^{G_{n}}(M_n\cdot)$ can be applied.

The first proposition says that $A^{(n)}_{p,\beta_p}$ can be ``bounded" by $\mytilde{\calE}_{p}^{G_{n}}(M_n\cdot)$.

\begin{myprop}\label{prop1}
There exists some positive constant $C_{\mathrm{P1}}$ such that
$$A^{(n)}_{p,\beta_p}(f)\le C_{\mathrm{P1}}\sup_{l\ge n}\mytilde{\calE}_{p}^{G_{l}}(M_lf)$$
for any $n\ge1$, $f\in\calF_p$.
\end{myprop}

\begin{proof}
By Lemma \ref{lem_lem512} and Lemma \ref{lem_eqn519}, we have
\begin{align*}
A^{(n)}_{p,\beta_p}(f)&\le c_{5.19}\left(\mytilde{\calE}_p^{G_n}(M_nf)+a^{\beta_pn}\sum_{w\in W_n}\int_{K_w}|f-M_nf(w)|^p\md\mu\right)\\
&\le c_{5.19}\left(\mytilde{\calE}_p^{G_n}(M_nf)+C_{\mathrm{PI}\text{-}\mathrm{KS}}\varliminf_{l\to+\infty}\mytilde{\calE}_{p}^{G_{n+l}}(M_{n+l}f)\right)\\
&\le \left(c_{5.19}\left(1+C_{\mathrm{PI}\text{-}\mathrm{KS}}\right)\right)\sup_{l\ge n}\mytilde{\calE}_{p}^{G_{l}}(M_lf).
\end{align*}
\end{proof}

The remaining three propositions are to show $\mytilde{\calE}_{p}^{G_{n}}(M_n\cdot)$ can be ``bounded" by $A^{(n)}_{p,\beta_p}$. However, a similar ``bound" to Proposition \ref{prop1} of the form
\begin{equation}\label{eqn_EsupA}
\mytilde{\calE}_{p}^{G_{n}}(M_nf)\le C\sup_{l\ge n}A^{(l)}_{p,\beta_p}(f),
\end{equation}
see \cite[Equation (5.20)]{Shi24}, is not adequate to prove our main result, hence some improvements are essentially needed.

The second proposition can be regarded as some improvement of \cite[Lemma 5.14]{Shi24}. Although it was mentioned at the beginning of its proof that its method was from \cite{GY19} by Grigor'yan and the author, the original idea was indeed from Jonsson \cite{Jon96} in 1996. Here we use an idea recently given by Gao, Yu and Zhang \cite{GYZ23} to do the improvement. Roughly speaking, the idea from \cite{GYZ23} is simply to divide integral in a ball into a summation of integrals over annuli.

For the simplicity of notation, for $n\ge1$, $f\in L^p(K;\mu)$, denote
$$A_n(f)=\int_K\int_{B(x,ca^{-n})\backslash B(x,ca^{-(n+1)})}|f(x)-f(y)|^p\md\mu(y)\md\mu(x),$$
then
$$A^{(n)}_{p,\beta_p}(f)=a^{(\alpha+\beta_p)n}\sum_{j=0}^\infty A_{n+j}(f)\ge a^{(\alpha+\beta_p)n}A_n(f).$$

Let $k\ge1$ be a fixed integer satisfying
\begin{equation}\label{eqn_k}
2^{p-1}<a^{(\beta_p-\alpha)k}.
\end{equation}
For $p>\mathrm{dim}_{\mathrm{ARC}}(K,d)\ge1$, since $\beta_p>\alpha$, such an integer $k\ge1$ always exists.

\begin{myprop}\label{prop2}
There exists some positive constant $C_{\mathrm{P2}}$ such that
$$a^{{\beta_p}n}\sum_{w\in W_n}\int_{K_w}|f-M_nf(w)|^p\md\mu\le C_{\mathrm{P2}}a^{(\alpha+\beta_p)n}\sum_{j=0}^\infty\left(2^{\frac{p-1}{k}}a^{2\alpha}\right)^jA_{n+j}(f)$$
for any $n\ge1$, $f\in\calF_p$.
\end{myprop}

\begin{proof}
Let $l\ge1$ be an integer. Fix $w\in W_n$, since $f\in\calF_p$ is continuous, there exists $x_w\in K_w$ such that $f(x_w)=M_nf(w)$, there exist $w_{n+1},w_{n+2},\ldots\in S$ such that $\{x_w\}=K_{ww_{n+1}w_{n+2}\ldots}$. Let
\begin{align*}
w^{(0)}&=w\in W_n,\\
w^{(1)}&=ww_{n+1}\ldots w_{n+k}\in W_{n+k},\\
w^{(2)}&=ww_{n+1}\ldots w_{n+k}w_{n+k+1}\ldots w_{n+2k}\in W_{n+2k},\\
&\ldots\\
w^{(l)}&=ww_{n+1}\ldots w_{n+k}w_{n+k+1}\ldots w_{n+2k}\ldots w_{n+k(l-1)+1}\ldots w_{n+kl}\in W_{n+kl}.
\end{align*}
For any $i=0,1,\ldots,l$, for any $x^{(i)}\in K_{w^{(i)}}$, by the H\"older inequality, we have
\begin{align*}\\
&|f(x^{(0)})-f(x_w)|^p\le2^{p-1}|f(x^{(0)})-f(x^{(l)})|^p+2^{p-1}|f(x^{(l)})-f(x_w)|^p\\
&\le2^{p-1}\left(2^{p-1}|f(x^{(0)})-f(x^{(1)})|^p+2^{p-1}|f(x^{(1)})-f(x^{(l)})|^p\right)+2^{p-1}|f(x^{(l)})-f(x_w)|^p\\
&=2^{2(p-1)}|f(x^{(0)})-f(x^{(1)})|^p+2^{2(p-1)}|f(x^{(1)})-f(x^{(l)})|^p+2^{p-1}|f(x^{(l)})-f(x_w)|^p\\
&\le2^{2(p-1)}|f(x^{(0)})-f(x^{(1)})|^p+2^{2(p-1)}\left(2^{p-1}|f(x^{(1)})-f(x^{(2)})|^p+2^{p-1}|f(x^{(2)})-f(x^{(l)})|^p\right)\\
&\quad+2^{p-1}|f(x^{(l)})-f(x_w)|^p\\
&=2^{2(p-1)}|f(x^{(0)})-f(x^{(1)})|^p+2^{3(p-1)}|f(x^{(1)})-f(x^{(2)})|^p+2^{3(p-1)}|f(x^{(2)})-f(x^{(l)})|^p\\
&\quad+2^{p-1}|f(x^{(l)})-f(x_w)|^p\\
&\le\ldots\le2^{2(p-1)}\sum_{i=0}^{l-1}2^{(p-1)i}|f(x^{(i)})-f(x^{(i+1)})|^p+2^{p-1}|f(x^{(l)})-f(x_w)|^p.
\end{align*}
Taking integration with respect to $x^{(0)}\in K_{w^{(0)}}$, $x^{(1)}\in K_{w^{(1)}}$, \ldots, $x^{(l)}\in K_{w^{(l)}}$ and dividing by $\mu(K_{w^{(0)}})$, $\mu(K_{w^{(1)}})$, \ldots, $\mu(K_{w^{(l)}})$, we have
\begin{align*}
&\frac{1}{\mu(K_w)}\int_{K_w}|f-M_nf(w)|^p\md\mu=\dashint_{K_{w^{(0)}}}|f(x^{(0)})-f(x_w)|^p\md\mu(x^{(0)})\\
&\le2^{2(p-1)}\sum_{i=0}^{l-1}2^{(p-1)i}\dashint_{K_{w^{(i)}}}\dashint_{K_{w^{(i+1)}}}|f(x^{(i)})-f(x^{(i+1)})|^p\md\mu(x^{(i+1)})\md\mu(x^{(i)})\\
&\quad+2^{p-1}\dashint_{K_{w^{(l)}}}|f(x^{(l)})-f(x_w)|^p\md\mu(x^{(l)}).
\end{align*}
For the first term, since $K_{w^{(i+1)}}\subseteq K_{w^{(i)}}$, by the choice Equation (\ref{eqn_smallc}) of the constant $c$, we have
\begin{align*}
&\dashint_{K_{w^{(i)}}}\dashint_{K_{w^{(i+1)}}}|f(x^{(i)})-f(x^{(i+1)})|^p\md\mu(x^{(i+1)})\md\mu(x^{(i)})\\
&\le\frac{1}{\mu(K_{w^{(i)}})\mu(K_{w^{(i+1)}})}\int_{K_{w^{(i)}}}\int_{B(x^{(i)},ca^{-(n+ki)})}|f(x^{(i)})-f(x^{(i+1)})|^p\md\mu(x^{(i+1)})\md\mu(x^{(i)})\\
&=a^{\alpha k}a^{2\alpha(n+ki)}\int_{K_{w^{(i)}}}\int_{B(x^{(i)},ca^{-(n+ki)})}|f(x^{(i)})-f(x^{(i+1)})|^p\md\mu(x^{(i+1)})\md\mu(x^{(i)}).
\end{align*}
For the second term, since $x^{(l)},x_w\in K_{w^{(l)}}$, by the choice Equation (\ref{eqn_smallc}) of the constant $c$ again, and by Lemma \ref{lem_lem513}, we have
\begin{align*}
&\dashint_{K_{w^{(l)}}}|f(x^{(l)})-f(x_w)|^p\md\mu(x^{(l)})\le\dashint_{K_{w^{(l)}}}\left(C_{5.13}d(x^{(l)},x_w)^{\beta_p-\alpha}\sup_{n\ge1}A^{(n)}_{p,\beta_p}(f)\right)\md\mu(x^{(l)})\\
&\le C_{5.13}(ca^{-(n+kl)})^{\beta_p-\alpha}\sup_{n\ge1}A^{(n)}_{p,\beta_p}(f)=\left(C_{5.13}c^{\beta_p-\alpha}\right)a^{-(\beta_p-\alpha)(n+kl)}\sup_{n\ge1}A^{(n)}_{p,\beta_p}(f).
\end{align*}
Hence
\begin{align*}
&\frac{1}{\mu(K_w)}\int_{K_w}|f-M_nf(w)|^p\md\mu\\
&\le2^{2(p-1)}\sum_{i=0}^{l-1}2^{(p-1)i}a^{\alpha k}a^{2\alpha(n+ki)}\\
&\quad\cdot\int_{K_{w^{(i)}}}\int_{B(x^{(i)},ca^{-(n+ki)})}|f(x^{(i)})-f(x^{(i+1)})|^p\md\mu(x^{(i+1)})\md\mu(x^{(i)})\\
&\quad+2^{p-1}\left(C_{5.13}c^{\beta_p-\alpha}\right)a^{-(\beta_p-\alpha)(n+kl)}\sup_{n\ge1}A^{(n)}_{p,\beta_p}(f),
\end{align*}
hence
\begin{align*}
&\int_{K_w}|f-M_nf(w)|^p\md\mu\\
&\le2^{2(p-1)}\sum_{i=0}^{l-1}2^{(p-1)i}a^{\alpha k}a^{\alpha n}a^{2\alpha ki}\\
&\quad\cdot\int_{K_{w^{(i)}}}\int_{B(x^{(i)},ca^{-(n+ki)})}|f(x^{(i)})-f(x^{(i+1)})|^p\md\mu(x^{(i+1)})\md\mu(x^{(i)})\\
&\quad+\left(2^{p-1}C_{5.13}c^{\beta_p-\alpha}\right)a^{-\alpha n}a^{-(\beta_p-\alpha)(n+kl)}\sup_{n\ge1}A^{(n)}_{p,\beta_p}(f),
\end{align*}
hence
\begin{align*}
&\sum_{w\in W_n}\int_{K_w}|f-M_nf(w)|^p\md\mu\\
&\le2^{2(p-1)}\sum_{i=0}^{l-1}2^{(p-1)i}a^{\alpha k}a^{\alpha n}a^{2\alpha ki}\\
&\quad\cdot\sum_{w\in W_n}\int_{K_{w^{(i)}}}\int_{B(x^{(i)},ca^{-(n+ki)})}|f(x^{(i)})-f(x^{(i+1)})|^p\md\mu(x^{(i+1)})\md\mu(x^{(i)})\\
&\quad+\left(2^{p-1}C_{5.13}c^{\beta_p-\alpha}\right)a^{-(\beta_p-\alpha)(n+kl)}\sup_{n\ge1}A^{(n)}_{p,\beta_p}(f)\\
&\le2^{2(p-1)}\sum_{i=0}^{l-1}2^{(p-1)i}a^{\alpha k}a^{\alpha n}a^{2\alpha ki}\int_{K}\int_{B(x,ca^{-(n+ki)})}|f(x)-f(y)|^p\md\mu(y)\md\mu(x)\\
&\quad+\left(2^{p-1}C_{5.13}c^{\beta_p-\alpha}\right)a^{-(\beta_p-\alpha)(n+kl)}\sup_{n\ge1}A^{(n)}_{p,\beta_p}(f)\\
&\le2^{2(p-1)}a^{\alpha k}a^{\alpha n}\sum_{i=0}^{\infty}2^{(p-1)i}a^{2\alpha ki}\sum_{j=n+ki}^\infty A_j(f)\\
&\quad+\left(2^{p-1}C_{5.13}c^{\beta_p-\alpha}\right)a^{-(\beta_p-\alpha)(n+kl)}\sup_{n\ge1}A^{(n)}_{p,\beta_p}(f).
\end{align*}
For the first term, by changing the order of summations, we have
\begin{align*}
&\sum_{i=0}^{\infty}2^{(p-1)i}a^{2\alpha ki}\sum_{j=n+ki}^\infty A_j(f)=\sum_{j=n}^{\infty}\left(\sum_{i=0}^{\left[\frac{j-n}{k}\right]}2^{(p-1)i}a^{2\alpha ki}\right)A_j(f)\\
&\le\frac{1}{2^{p-1}a^{2\alpha k}-1}\sum_{j=n}^\infty\left(2^{p-1}a^{2\alpha k}\right)^{\left[\frac{j-n}{k}\right]+1}A_j(f)\le\frac{2^{p-1}a^{2\alpha k}}{2^{p-1}a^{2\alpha k}-1}\sum_{j=n}^\infty\left(2^{p-1}a^{2\alpha k}\right)^{\frac{j-n}{k}}A_j(f)\\
&=\frac{2^{p-1}a^{2\alpha k}}{2^{p-1}a^{2\alpha k}-1}\sum_{j=n}^\infty\left(2^{\frac{p-1}{k}}a^{2\alpha}\right)^{j-n}A_j(f)=\frac{2^{p-1}a^{2\alpha k}}{2^{p-1}a^{2\alpha k}-1}\sum_{j=0}^\infty\left(2^{\frac{p-1}{k}}a^{2\alpha}\right)^{j}A_{n+j}(f),
\end{align*}
where $[x]$ is the greatest integer less than or equal to $x\in\R$. For the second term, since $\beta_p>\alpha$, letting $l\to+\infty$, we have $a^{-(\beta_p-\alpha)(n+kl)}\to0$. Since $f\in\calF_p$, we have $\sup_{n\ge1}A^{(n)}_{p,\beta_p}(f)<+\infty$. Hence the second term tends to 0 as $l\to+\infty$. Hence letting $l\to+\infty$, we have
\begin{align*}
&\sum_{w\in W_n}\int_{K_w}|f-M_nf(w)|^p\md\mu\\
&\le2^{2(p-1)}a^{\alpha k}a^{\alpha n}\frac{2^{p-1}a^{2\alpha k}}{2^{p-1}a^{2\alpha k}-1}\sum_{j=0}^\infty\left(2^{\frac{p-1}{k}}a^{2\alpha}\right)^{j}A_{n+j}(f)\\
&=\left(\frac{2^{3(p-1)}a^{3\alpha k}}{2^{p-1}a^{2\alpha k}-1}\right)a^{\alpha n}\sum_{j=0}^\infty\left(2^{\frac{p-1}{k}}a^{2\alpha}\right)^{j}A_{n+j}(f),
\end{align*}
which is our desired result.
\end{proof}

Instead of Equation (\ref{eqn_EsupA}), $\mytilde{\calE}_p^{G_n}(M_n\cdot)$ can be ``bounded" by $A_n$ as follows.

\begin{myprop}\label{prop3}
There exists some positive constant $C_{\mathrm{P3}}$ such that
$$\mytilde{\calE}_p^{G_n}(M_nf)\le C_{\mathrm{P3}}a^{(\alpha+\beta_p)n}\sum_{j=0}^\infty\left(2^{\frac{p-1}{k}}a^{2\alpha}\right)^{j}A_{n+j}(f)$$
for any $n\ge1$, $f\in\calF_p$.
\end{myprop}

\begin{proof}
Recall that
$$\mytilde{\calE}_p^{G_n}(M_nf)=\rho_p^n\calE_p^{G_n}(M_nf)=a^{(\beta_p-\alpha)n}\sum_{(w,v)\in E_n}|M_nf(w)-M_nf(v)|^p.$$
For any $(w,v)\in E_n$, for any $x\in K_w$, $y\in K_v$, by the H\"older inequality, we have
$$|M_nf(w)-M_nf(v)|^p\le3^{p-1}\left(|M_nf(w)-f(x)|^p+|f(x)-f(y)|^p+|M_nf(v)-f(y)|^p\right).$$
Taking integration with respect to $x\in K_w$, $y\in K_v$ and dividing by $\mu(K_w)$, $\mu(K_v)$, by the choice Equation (\ref{eqn_smallc}) of the constant $c$, we have
\begin{align*}
&|M_nf(w)-M_nf(v)|^p\\
&\le3^{p-1}\left(\dashint_{K_w}|f-M_nf(w)|^p\md\mu+\dashint_{K_w}\dashint_{K_v}|f(x)-f(y)|^p\md\mu(y)\md\mu(x)\right.\\
&\quad\left.+\dashint_{K_v}|f-M_nf(v)|^p\md\mu\right)\\
&\le3^{p-1}\left(a^{\alpha n}\int_{K_w}|f-M_nf(w)|^p\md\mu+a^{2\alpha n}\int_{K_w}\int_{B(x,ca^{-n})}|f(x)-f(y)|^p\md\mu(y)\md\mu(x)\right.\\
&\quad\left.+a^{\alpha n}\int_{K_v}|f-M_nf(v)|^p\md\mu\right).
\end{align*}
Taking summation with respect to $(w,v)\in E_n$, by the choice Equation (\ref{eqn_L*}) of the constant $L_*$, we have
\begin{align*}
&\sum_{(w,v)\in E_n}|M_nf(w)-M_nf(v)|^p\\
&\le3^{p-1}\left(2a^{\alpha n}\sum_{w\in W_n}\int_{K_w}|f-M_nf(w)|^p\md\mu\right.\\
&\quad\left.+L_*a^{2\alpha n}\int_K\int_{B(x,ca^{-n})}|f(x)-f(y)|^p\md\mu(y)\md\mu(x)\right)\\
&=3^{p-1}\left(2a^{\alpha n}\sum_{w\in W_n}\int_{K_w}|f-M_nf(w)|^p\md\mu+L_*a^{2\alpha n}\sum_{j=0}^\infty A_{n+j}(f)\right).
\end{align*}
By Proposition \ref{prop2}, we have
\begin{align*}
&\mytilde{\calE}_p^{G_n}(M_nf)\\
&\le3^{p-1}\left(2a^{\beta_p n}\sum_{w\in W_n}\int_{K_w}|f-M_nf(w)|^p\md\mu+L_*a^{(\alpha+\beta_p)n}\sum_{j=0}^\infty A_{n+j}(f)\right)\\
&\le3^{p-1}\left(2C_{\mathrm{P2}}a^{(\alpha+\beta_p)n}\sum_{j=0}^\infty\left(2^{\frac{p-1}{k}}a^{2\alpha}\right)^jA_{n+j}(f)+L_*a^{(\alpha+\beta_p)n}\sum_{j=0}^\infty A_{n+j}(f)\right)\\
&\le\left(3^{p-1}(2C_{\mathrm{P2}}+L_*)\right)a^{(\alpha+\beta_p)n}\sum_{j=0}^\infty\left(2^{\frac{p-1}{k}}a^{2\alpha}\right)^jA_{n+j}(f).
\end{align*}
\end{proof}

With the help of $A_n$, $\mytilde{\calE}_p^{G_n}(M_n\cdot)$ can be ``bounded" by $A^{(n)}_{p,\beta_p}$ in the following result.

\begin{myprop}\label{prop4}
There exists some positive constant $C_{\mathrm{P4}}$ such that
$$\varliminf_{n\to+\infty}\mytilde{\calE}_p^{G_n}(M_nf)\le C_{\mathrm{P4}}\varliminf_{n\to+\infty}A^{(n)}_{p,\beta_p}(f)$$
for any $f\in\calF_p$.
\end{myprop}

\begin{proof}
Let $J\ge1$ be an integer to be determined. For any $n\ge1$, by Proposition \ref{prop3}, we have
$$\mytilde{\calE}_p^{G_n}(M_nf)\le C_{\mathrm{P3}}a^{(\alpha+\beta_p)n}\left(\sum_{j=0}^J+\sum_{j=J+1}^\infty\right)\left(2^{\frac{p-1}{k}}a^{2\alpha}\right)^{j}A_{n+j}(f).$$
Recall that
$$A^{(n)}_{p,\beta_p}(f)=a^{(\alpha+\beta_p)n}\sum_{j=0}^\infty A_{n+j}(f)\ge a^{(\alpha+\beta_p)n}A_n(f).$$
Hence
\begin{align*}
&a^{(\alpha+\beta_p)n}\sum_{j=0}^J\left(2^{\frac{p-1}{k}}a^{2\alpha}\right)^{j}A_{n+j}(f)\le a^{(\alpha+\beta_p)n}\sum_{j=0}^J\left(2^{\frac{p-1}{k}}a^{2\alpha}\right)^{J}A_{n+j}(f)\\
&\le\left(2^{\frac{p-1}{k}}a^{2\alpha}\right)^{J}a^{(\alpha+\beta_p)n}\sum_{j=0}^\infty A_{n+j}(f)=\left(2^{\frac{p-1}{k}}a^{2\alpha}\right)^{J}A^{(n)}_{p,\beta_p}(f),
\end{align*}
and
\begin{align*}
&a^{(\alpha+\beta_p)n}\sum_{j=J+1}^\infty\left(2^{\frac{p-1}{k}}a^{2\alpha}\right)^{j}A_{n+j}(f)\\
&\le a^{(\alpha+\beta_p)n}\sum_{j=J+1}^\infty\left(2^{\frac{p-1}{k}}a^{2\alpha}\right)^{j}a^{-(\alpha+\beta_p)(n+j)}A^{(n+j)}_{p,\beta_p}(f)\\
&=\sum_{j=J+1}^\infty\left(2^{\frac{p-1}{k}}a^{-(\beta_p-\alpha)}\right)^{j}A^{(n+j)}_{p,\beta_p}(f).
\end{align*}
By the choice Equation (\ref{eqn_k}) of the integer $k\ge1$, we have $2^{\frac{p-1}{k}}a^{-(\beta_p-\alpha)}\in(0,1)$. By Lemma \ref{lem_WM} and Proposition \ref{prop1}, we have
$$A^{(n+j)}_{p,\beta_p}(f)\le C_{\mathrm{P1}}\sup_{l\ge n+j}\mytilde{\calE}_p^{G_l}(M_lf)\le C_{\mathrm{P1}}C_{\mathrm{WM}}\varliminf_{n\to+\infty}\mytilde{\calE}_p^{G_n}(M_nf).$$
Hence
\begin{align}
&\mytilde{\calE}_p^{G_n}(M_nf)\nonumber\\
&\le C_{\mathrm{P3}}\left(2^{\frac{p-1}{k}}a^{2\alpha}\right)^{J}A^{(n)}_{p,\beta_p}(f)+C_{\mathrm{P3}}\sum_{j=J+1}^\infty\left(2^{\frac{p-1}{k}}a^{-(\beta_p-\alpha)}\right)^{j}C_{\mathrm{P1}}C_{\mathrm{WM}}\varliminf_{n\to+\infty}\mytilde{\calE}_p^{G_n}(M_nf)\nonumber\\
&=C_{\mathrm{P3}}\left(2^{\frac{p-1}{k}}a^{2\alpha}\right)^{J}A^{(n)}_{p,\beta_p}(f)+\left(C_{\mathrm{P1}}C_{\mathrm{P3}}C_{\mathrm{WM}}\frac{\left(2^{\frac{p-1}{k}}a^{-(\beta_p-\alpha)}\right)^{J+1}}{1-2^{\frac{p-1}{k}}a^{-(\beta_p-\alpha)}}\right)\varliminf_{n\to+\infty}\mytilde{\calE}_p^{G_n}(M_nf)\label{eqn_EliminfE}.
\end{align}
Firstly, since $2^{\frac{p-1}{k}}a^{-(\beta_p-\alpha)}\in(0,1)$, choosing $J\ge1$ sufficiently large, we have
$$C_{\mathrm{P1}}C_{\mathrm{P3}}C_{\mathrm{WM}}\frac{\left(2^{\frac{p-1}{k}}a^{-(\beta_p-\alpha)}\right)^{J+1}}{1-2^{\frac{p-1}{k}}a^{-(\beta_p-\alpha)}}\le\frac{1}{2},$$
then
$$\mytilde{\calE}_p^{G_n}(M_nf)\le C_{\mathrm{P3}}\left(2^{\frac{p-1}{k}}a^{2\alpha}\right)^{J}A^{(n)}_{p,\beta_p}(f)+\frac{1}{2}\varliminf_{n\to+\infty}\mytilde{\calE}_p^{G_n}(M_nf).$$
Secondly, letting $n\to+\infty$, we have
$$\varliminf_{n\to+\infty}\mytilde{\calE}_p^{G_n}(M_nf)\le C_{\mathrm{P3}}\left(2^{\frac{p-1}{k}}a^{2\alpha}\right)^{J}\varliminf_{n\to+\infty}A^{(n)}_{p,\beta_p}(f)+\frac{1}{2}\varliminf_{n\to+\infty}\mytilde{\calE}_p^{G_n}(M_nf),$$
hence
$$\varliminf_{n\to+\infty}\mytilde{\calE}_p^{G_n}(M_nf)\le\left(2C_{\mathrm{P3}}\left(2^{\frac{p-1}{k}}a^{2\alpha}\right)^{J}\right)\varliminf_{n\to+\infty}A^{(n)}_{p,\beta_p}(f).$$
\end{proof}

\begin{myrmk}
By Equation (\ref{eqn_EliminfE}), we can indeed obtain that for any $\veps\in(0,1)$, there exists some positive constant $C(\veps)$ depending on $\veps$ such that
$$\mytilde{\calE}_p^{G_n}(M_nf)\le C(\veps)A^{(n)}_{p,\beta_p}(f)+\veps\varliminf_{n\to+\infty}\mytilde{\calE}_p^{G_n}(M_nf)$$
for any $n\ge1$, $f\in\calF_p$.
\end{myrmk}

Now we give the proof of Theorem \ref{thm_main} as follows.

\begin{proof}[Proof of Theorem \ref{thm_main}]
For any $n\ge1$, we have
\begin{align*}
&A^{(n)}_{p,\beta_p}(f)\le C_{\mathrm{P1}}\sup_{l\ge n}\mytilde{\calE}_{p}^{G_{l}}(M_lf)\\
&\le C_{\mathrm{P1}}C_{\mathrm{WM}}\varliminf_{n\to+\infty}\mytilde{\calE}_{p}^{G_{n}}(M_nf)\\
&\le C_{\mathrm{P1}}C_{\mathrm{WM}}C_{\mathrm{P4}}\varliminf_{n\to+\infty}A^{(n)}_{p,\beta_p}(f),
\end{align*}
where we use Proposition \ref{prop1} in the first inequality, the weak monotonicity result Lemma \ref{lem_WM} in the second inequality and Proposition \ref{prop4} in the third inequality. Taking supremum with respect to $n\ge1$, we have
$$\sup_{n\ge1}A^{(n)}_{p,\beta_p}(f)\le\left(C_{\mathrm{P1}}C_{\mathrm{P4}}C_{\mathrm{WM}}\right)\varliminf_{n\to+\infty}A^{(n)}_{p,\beta_p}(f).$$
\end{proof}

\bibliographystyle{plain}

\end{document}